\documentclass[12pt]{amsart} 

\usepackage[utf8]{inputenc}
\usepackage{amsmath, amsthm, amsfonts, amssymb}
\usepackage{amsaddr}
\usepackage{hyperref} 
\usepackage{xcolor}
\usepackage{float}
\usepackage[shortlabels]{enumitem}
\usepackage{tikz}
\usepackage{diagbox}

\usepackage{pgfplots}
\pgfplotsset{compat=1.15}
\usepackage{mathrsfs}
\usetikzlibrary{arrows}
\usetikzlibrary[patterns]

\definecolor{light-gray}{gray}{0.95}

\newtheorem{theorem}{Theorem}
\newtheorem{lemma}[theorem]{Lemma}
\newtheorem{coro}[theorem]{Corollary}

\theoremstyle{definition}
\newtheorem{remark}[theorem]{Remark}
\newtheorem{alg}[theorem]{Algorithm}

\newcommand{\eps}{\varepsilon}
\renewcommand{\phi}{\varphi}
\newcommand{\e}{\varepsilon}

\newcommand{\N}{\mathbb{N}}
\newcommand{\Z}{\mathbb{Z}}
\newcommand{\R}{\mathbb{R}}

\newcommand{\be}{\begin{equation}}
\newcommand{\ee}{\end{equation}}
\newcommand{\co}{\colon\, }

\newcommand{\contada}{\text{\rm cont}}
\newcommand{\contnon}{\text{\rm cont-non}}

\DeclareMathOperator{\Lip}{Lip}
\DeclareMathOperator{\diam}{diam}
\DeclareMathOperator{\dist}{dist}


\newcommand*\oline[1]{%
  \vbox{%
    \hrule height 0.5pt
    \kern0.25ex
    \hbox{%
      \kern-0.1em
      \ifmmode#1\else\ensuremath{#1}\fi
      \kern-0.1em
    }
  }
}

\title[How many measurements?]{
How many continuous measurements are needed to learn a vector?
}

\author{
David Krieg, Erich Novak and Mario Ullrich
}

\date{\today}


\begin{document}

\begin{abstract}
One can recover vectors from $\R^m$ with arbitrary precision, 
using only 
$\lceil \log_2(m)\rceil +1$
continuous measurements that are 
chosen adaptively. 
This surprising result is explained and discussed,
and we present applications to 
infinite-dimensional 
approximation problems. 
\end{abstract}

\maketitle


\section{Introduction}
\label{sec:intro}

Assume that you want to recover an arbitrary 
vector $x \in \R^m$, 
up to some error $\e >0$
in some norm $\Vert\cdot\Vert$, where 
$m \in \N$ can be large. 
You know nothing about $x$, you can only 
compute certain measurements 
$\lambda_1(x),\dots,\lambda_n(x) \in \R$.
How many measurements do you need?

\medskip

If the measurement maps $\lambda_1,\hdots,
\lambda_n$ are linear, you will quickly 
realize that less than $m$ measurements 
are useless. Let us denote 
$N=(\lambda_1,\hdots,\lambda_n)$ with $n<m$. 
Then, no matter the value of the information 
$y=N(x)$ that you obtain, there always is a whole 
affine subspace $V$ of $\R^m$ such that $N(v)=y$ 
for all $v \in V$. So no matter how you choose 
your approximation $\hat x = \Phi(y) \in\R^m$, you may 
be arbitrarily far away from the 
true value of $x$.
The same is true for any continuous measurement map $N \colon \R^m \to \R^n$. 
This follows from the Borsuk-Ulam theorem:
For any $R>0$, there are $x$ and $\tilde x$ at 
distance $2R$ such that $N(x)=N(\tilde x)$, i.e., 
you cannot distinguish between those vectors 
and hence cannot guarantee an error less than $R$.

\medskip

But why should we fix the measurement maps 
$\lambda_j$ in advance? Would it not be better to use the already obtained information 
$\lambda_1(x)$,\ldots, $\lambda_{j-1}(x)$ and 
choose $\lambda_j$ based on it? We call such 
measurements adaptive. 
Unfortunately, for linear functionals $\lambda_j$, 
the fact remains that 
less than $m$ measurements are useless, 
also if we choose them adaptively.

Indeed, assume that you obtain the information $y=N(x)$. Then, via $y$, 
you also know which functionals $\lambda_1,\hdots,\lambda_n$ have been chosen and
there is a whole affine subspace $V\subset \R^m$ such that $\lambda_j(v)=y_j$ for all $j\le n$. Thus, for any $v\in V$, you would have chosen the 
same functionals $\lambda_j$ and obtained the same 
information $y=N(v)$, which means that you cannot 
distinguish between all the elements of $V$.

\medskip

Thus, less than $m$ non-adaptive linear measurements are useless and neither 
adaption nor allowing non-linear, 
continuous measurements
improve the 
situation.
But what happens if we consider both together, i.e., if we allow adaptively chosen continuous 
measurements?
Can we then achieve something with less than $m$ measurements? 

\medskip

Surprisingly, the answer is yes.
In fact, 
then the number $n$ of measurements 
can be chosen
exponentially smaller than  the 
dimension $m$ of the space. 
We will show that
 \[
 n (m) := \lceil \log_2(m)\rceil +1 
 \]
 continuous adaptive measurements suffice:

 \begin{theorem}\label{thm:main}
Let $m\in\N$ and $\eps>0$. The algorithm $R_m^\eps$ described below 
uses at most $n(m)$
adaptive, continuous 
measurements
and satisfies for all $x\in\R^m$ that
\[
 \Vert x - R_m^\eps(x) \Vert \,\le\, \eps.
\]
\end{theorem}

 \medskip

If we allow $n(m+1)$ instead of $n(m)$ measurements,
the measurements from Theorem~\ref{thm:main} can be chosen with Lipschitz-constant
at most one.
We define the algorithm and prove the result in Section~\ref{sec:finite}.
We do not know whether the bound is optimal. 
From the Borsuk-Ulam Theorem 
we only know 
that one continuous measurement 
is not enough if $m>1$.

\medskip

Note that, in this paper, the term `algorithm' 
 can refer to any approximation procedure that processes no other information than 
 the allowed type of measurements, regardless of practical feasibility.
 Further discussion of the result and its applicability can be found in Section~\ref{sec:discussion}.
 In Section~\ref{sec:infinite}, we will explore some implications 
 of the above result for infinite-dimensional problems.

\section{Proof of Theorem~\ref{thm:main}}
\label{sec:finite}

In this section, we prove our main result, Theorem~\ref{thm:main}, by presenting an algorithm that achieves the bound.
We first show that $n(m+1)=\lceil \log_2(m+1)\rceil +1$ Lipschitz measurements suffice and afterwards modify the algorithm to reduce to $n(m)$ continuous measurements.

\medskip

\begin{alg}
For $m\in\N$ and $\eps>0$, the algorithm $R_m^\eps$ is defined as follows.
Consider a partition
$$
\R^m = \bigcup_{i \in \N} D_i, 
$$
a (coloring) map
$t\colon\N\to\{ 1,2, \dots , m+1 \}$,
and a constant $c>0$ with:
\begin{enumerate} 
\itemsep=1mm
\item 
$\diam (D_i) \le 1$ for each $i$;

\item 
if $i \not= j$ and $t(i)=t(j)$ then 
$\dist (D_i, D_j) \ge c$, 
\end{enumerate} 
where diameter ($\diam$) and distance ($\dist$) are with respect to the given norm on $\R^m$. 
We will show later that such a partition exists. 
Less than $m+1$ colors do not work, see Remark~\ref{rem:color}.

To find 
an $\eps$-approximation of $x\in\R^m$ 
it is clearly enough to find 
an index $i^*$ 
such that $x$ is in the closure of $\eps D_{i^*}$. 
Define by 
$$
I_r := \{i\in\N\colon t(i) = r\} 
\quad\text{ and }\quad
E_r  = \bigcup_{ j\in I_r} \eps D_j
$$
the set of points with color $r$. 
A continuous measurement 
of the form 
$$
\lambda_{J} (x) = \dist \biggl( 
x, 
\bigcup_{j\in J}  E_{j} 
\biggr)  
$$
with $J\subset\{1,\dots,m+1\}$, 
tells us whether or not $x$ 
is contained in the closure of any of the sets $E_j$, $j\in J$.
We use
$
n = \lceil \log_2(m+1)\rceil = n(m+1) -1  
$
such functionals
and bisection to find a color 
$r^*=t(i^*)$ such that 
$x \in 
\oline{E}_{r^*}$.\footnote{We do not claim 
that $r$ and $i^*$ are  unique. 
}

Now we can determine a correct index $i^*$
with $x\in 
\eps \overline{D}_{i*}$
using   
any continuous functional $\lambda^*$ for which the images $\lambda^*(\eps\oline{ D_i})$ for $i\in I_{r^*}$ are pairwise disjoint. An example is given by
\begin{equation*}\label{eq:last-info}
\lambda^*  (x) = \max_{i\in I_{r^*}} \left\{\frac{c\eps}{2i} - \dist\left(x, \eps D_i\right)\right\}.  
\end{equation*}
Note that, for each $x\in\R^m$, 
at most one of the terms in the maximum is non-negative. 
Since $x$ is contained in the closure of $\eps D_{i^*}$ 
with $i^*=\frac{c\eps}{2\lambda^* (x)}$,
the output $R_m^\eps(x)$ of the algorithm
can be chosen as any element from $\eps D_{i^*}$.

All the functionals above have Lipschitz-constant equal to one
and, if we consider the $1$- or $\infty$-norm, 
they are even piecewise linear.
One may consider 
the functionals $\lambda_J$ 
as continuous membership oracles. 
\end{alg}

\medskip

It remains to find a partition 
$
\R^m = \bigcup_{i \in \N} D_i 
$
with the properties above. 
We only consider the 
maximum norm on $\R^m$.
The general case is obtained by 
the equivalence of all norms on $\R^m$.

The idea is to split $\R^m$ into unit cubes 
and give all $k$-dimensional facets the color $k$.
This does not quite work yet, since different facets of the same dimension $k$ touch.
However, if we add a small neighborhood to the $(k-1)$-dimensional facets,
then the $k$-dimensional facets do not touch any more, see the right hand side of Figure~\ref{fig:partition}. 

More precisely, 
let us denote by $F_k$ the $k$-dimensional facets, i.e., the set of vectors in $\R^m$ with at least $m-k$ integer components, and with $F_k^\delta$ for $\delta>0$ the closed $\delta$-neighborhood of $F_k$ in the maximum norm. We choose $1/2> \delta_0 > \dots > \delta_m >0$ and 
inductively define $C_0:= F_0^{\delta_0}$ and 
\[
 C_k := F_k^{\delta_k} \setminus \bigcup_{i< k} C_i, 
 \quad k=1,\hdots,m.
\]
We give $C_k$ the color $k$. 
The desired covering is obtained as the family $(D_i)_{i\in\N}$ of the connected components of each of the sets $C_k$, $0\le k \le m$.
It is easy to verify that this partition has the desired properties.
Indeed, for all $k\ge 1$, the connected components of $C_k$ are at a distance at least $\delta_{k-1}-\delta_k$.
%
The above partition is certainly not the only admissible one, another partition for $m=2$ is shown on the left hand side of Figure~\ref{fig:partition}.

 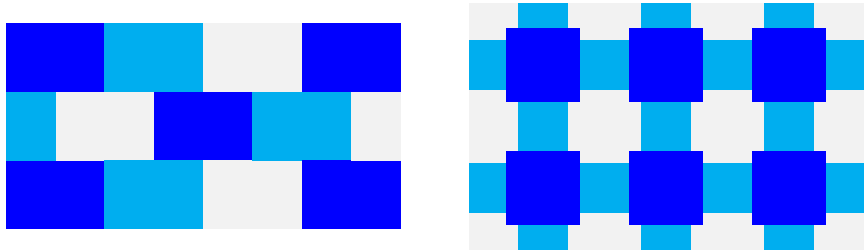
\begin{figure}[H]
 
\hfill
\resizebox{0.43\linewidth}{!}{
\begin{tikzpicture}[line cap=round,line join=round,>=triangle 45,x=0.9cm,y=1.25cm]
\clip(1,-.3) rectangle (9,4);
\fill[line width=2pt,color=blue,fill=blue,fill opacity=1] (1,1) -- (3,1) -- (3,0) -- (1,0) -- cycle;
\fill[line width=2pt,fill=light-gray] (2,2) -- (4,2) -- (4,1) -- (2,1) -- cycle;
\fill[line width=2pt,color=blue,fill=blue,fill opacity=1] (1,2) -- (3,2) -- (3,3) -- (1,3) -- cycle;
\fill[line width=2pt,color=cyan,fill=cyan,fill opacity=0.9] (3,3) -- (5,3) -- (5,2) -- (3,2) -- cycle;
\fill[line width=2pt,fill=light-gray] (5,3) -- (7,3) -- (7,2) -- (5,2) -- cycle;
\fill[line width=2pt,color=blue,fill=blue,fill opacity=1] (4,2) -- (6,2) -- (6,1) -- (4,1) -- cycle;
\fill[line width=2pt,color=cyan,fill=cyan,fill opacity=0.9] (3,1) -- (5,1) -- (5,0) -- (3,0) -- cycle;
\fill[line width=2pt,fill=light-gray] (5,1) -- (7,1) -- (7,0) -- (5,0) -- cycle;
\fill[line width=2pt,color=cyan,fill=cyan,fill opacity=0.9] (0,1) -- (0,2) -- (2,2) -- (2,1) -- cycle;
\fill[line width=2pt,color=cyan,fill=cyan,fill opacity=0.9] (6,2) -- (8,2) -- (8,1) -- (6,1) -- cycle;
\fill[line width=2pt,color=blue,fill=blue,fill opacity=1] (7,3) -- (9,3) -- (9,2) -- (7,2) -- cycle;
\fill[line width=2pt,color=blue,fill=blue,fill opacity=1] (7,1) -- (9,1) -- (9,0) -- (7,0) -- cycle; 
\fill[line width=2pt,fill=light-gray] (8,2) -- (10,2) -- (10,1) -- (8,1) -- cycle;
\end{tikzpicture}
}
\hfill
\resizebox{0.43\linewidth}{!}{
\begin{tikzpicture}[line cap=round,line join=round,>=triangle 45,x=2.2cm,y=2.2cm]
\clip(-6.6,0.5) rectangle (-3.4,2.5);
\fill[line width=2pt,fill=light-gray] (-7,0) -- (-3,0) -- (-3,3) -- (-7,3) -- cycle;
\fill[line width=10pt,color=cyan,fill=cyan,fill opacity=0.9] (-6.2,3.3) -- (-6.2,-0.3) -- (-5.8,-0.3) -- (-5.8,3.3) -- cycle;
\fill[line width=0pt,color=cyan,fill=cyan,fill opacity=0.9] (-5.2,3.4) -- (-5.2,-0.2) -- (-4.8,-0.2) -- (-4.8,3.4) -- cycle;
\fill[line width=0pt,color=cyan,fill=cyan,fill opacity=0.9] (-4.2,3.3) -- (-4.2,-0.3) -- (-3.8,-0.3) -- (-3.8,3.3) -- cycle;
\fill[line width=0pt,color=cyan,fill=cyan,fill opacity=0.9] (-7.4,2.2) -- (-7.4,1.8) -- (-2.6,1.8) -- (-2.6,2.2) -- cycle;
\fill[line width=0pt,color=cyan,fill=cyan,fill opacity=0.9] (-7.4,1.2) -- (-7.4,0.8) -- (-2.6,0.8) -- (-2.6,1.2) -- cycle;

\fill[line width=2pt,color=blue,fill=blue,fill opacity=1] (-6.3,2.3) -- (-6.3,1.7) -- (-5.7,1.7) -- (-5.7,2.3) -- cycle;
\fill[line width=2pt,color=blue,fill=blue,fill opacity=1] (-6.3,1.3) -- (-6.3,0.7) -- (-5.7,0.7) -- (-5.7,1.3) -- cycle;
\fill[line width=2pt,color=blue,fill=blue,fill opacity=1] (-5.3,2.3) -- (-5.3,1.7) -- (-4.7,1.7) -- (-4.7,2.3) -- cycle;
\fill[line width=2pt,color=blue,fill=blue,fill opacity=1] (-4.3,2.3) -- (-4.3,1.7) -- (-3.7,1.7) -- (-3.7,2.3) -- cycle;
\fill[line width=2pt,color=blue,fill=blue,fill opacity=1] (-5.3,1.3) -- (-5.3,0.7) -- (-4.7,0.7) -- (-4.7,1.3) -- cycle;
\fill[line width=2pt,color=blue,fill=blue,fill opacity=1] (-4.3,1.3) -- (-4.3,0.7) -- (-3.7,0.7) -- (-3.7,1.3) -- cycle;
\end{tikzpicture}
}
\hspace{2mm}
\hfill

\vspace{2mm}
\caption{
Two colorings of the plane with three colors. 
The second is generalized above to higher 
dimensions.
}
\label{fig:partition}
\end{figure}

\begin{remark}\label{rem:color}
Above we used a partition of $\R^m$ with $m+1$ colors. In general, an $N$-color partition of $\R^m$ with the two properties from above results in an algorithm using $\lceil
\log_2  N \rceil +1$ continuous measurements.
It is not possible to find a partition with less than $m+1$ colors. 
This is related to the Lebesgue covering dimension of $\R^m$ which is $m$,
see \cite[Chapter 2]{Pears}. 
Still, there might exist 
partitions with more 
colors that are easier to implement.
\end{remark}

We have thus proven that $n(m+1)$ Lipschitz-continuous measurements suffice to recover an $m$-dimensional vector up to the prescribed error tolerance.
We now observe that, actually,
the same number of continuous 
measurements suffices to recover an $(m+1)$-dimensional vector $x$ instead of only an $m$-dimensional vector.
One can use the same first $n(m+1)-1$ measurements and apply them to $x':=(x_1,...,x_m)$ in order to identify a correct color of $x'$.
Only the last measurement is modified: We choose $\lambda^*$ such that it takes distinct integer values on the boxes of the corresponding color and consider the measurement $\lambda^{**}(x)=\lambda^*(x')+h(x_{m+1})$, where $h$ is some continuous bijective mapping from $\R$ to $(0,1)$. Then the integer part of $\lambda^{**}(x)$ tells us the correct box for $x'$ while the fractional part of $\lambda^{**}(x)$ additionally tells us the coordinate $x_{m+1}$.

In other words, for all $m\ge 2$, we have that $n(m) = \lceil {\log_2(m)\rceil+1}$ continuous measurements suffice to learn an $m$-dimensional vector. Clearly, the statement is also true for $m=1$, in which case $n(m)=m$. 

\begin{remark}
Note that the \emph{full} algorithm $R_m^\eps$ given above is not Lip\-schitz continuous;
it takes values in a discrete set.
(There cannot be a continuous 
algorithm,  
the Borsuk-Ulam theorem 
does not allow it.)  
We do not think that this is an issue, because it seems quite natural to use algorithms
that allow 
some kind of discrete decision, 
such as bisection.  
Still, 
from Theorem~\ref{thm:main} we obtain that 
\[
\|R_m^\eps(x)-R_m^\eps(y)\|\le \|x-y\|+2\eps, 
\]
which might be considered as some kind of \emph{stability}. 
\end{remark}

\begin{remark}\label{rem:random} 
Strangely, if we allow randomness, a \emph{single} random continuous measurement is enough to get a small error with high probability. 
Indeed, let $\eps>0$ and $\delta\in(0,1)$. 
The goal is to have an error less than $\eps$ (e.g., in the maximum norm) with probability $1-\delta$.
Let $v$ be uniformly distributed in $(0,1)^m$ and consider the family of cubes with midpoints in $\eps \cdot (\Z^m+v)$ and side length $\eps\cdot (1-\delta)^{1/m}$.
Since the cubes are separated, we find a continuous measurement map $L\colon \R^m \to \R$ that has a constant and distinct value on each cube.
If $x\in\R^m$ belongs to any of the cubes, then the measurement $L(x)$ tells us the correct cube and we know $x$ up to an error less than $\eps$.
The probability that $x$ is in none of the cubes is $\delta$.
So, for each $x$, we make an error of at most $\eps$ with probability $1-\delta$.
In fact, it is even possible to make an error at most $\varepsilon$ with probability equal to one by taking two random continuous measurements instead of one; we thank an anonymous referee for sharing this observation.
Both observations are in contrast to linear measurements where randomization does not help, which follows from (2) in \cite{Novak92}.
\end{remark}

\medskip
\goodbreak

\section{Discussion}\label{sec:discussion}
Did we really construct a clever algorithm 
for the recovery 
of $x \in \R^m$ if $m$ is large? 
We made two basic assumptions: 
\begin{enumerate} 
\item[1.] 
If $\lambda \colon \R^m \to \R$
is continuous then we can compute values 
$\lambda(x)$;
i.e., all continuous functionals are 
admissible as 
information.
 \item[2.]
The information can be chosen adaptively, i.e., 
$\lambda_{k+1}$ may depend on 
the (already computed) values 
$y_i = \lambda_i (x)$ 
for $i=1,2, \dots , k$. 
\end{enumerate} 

The use of adaption is widespread and 
can be easily implemented on a computer,
we do not see a problem here. 
The use of arbitrary continuous information 
functionals $\lambda \colon \R^m \to \R$
is certainly more problematic. 
On the one hand, 
there seems to be an increased interest recently in arbitrary continuous information. We refer to the notion of
manifold widths and other nonlinear widths 
which are discussed in the next section.
These widths
are often motivated as benchmarks for corresponding numerical procedures. 
On the other hand, 
it is shown in this paper that
continuous information allows an 
exponential speed-up over linear information. 
We think this is unrealistic and therefore
guess that the notion of 
arbitrary continuous information 
is too general. 

In the literature on information-based
complexity
it is usually assumed that only 
(certain or all) \emph{linear} functionals 
are admissible. 
Then we have two possibilities 
to generalize over 
linear information in the form of a linear mapping $N\colon \R^m \to \R^n$: 
\begin{itemize} 
\item 
we may use \emph{adaption}, 
as in this paper; 
\item 
we may use \emph{randomization}, 
not discussed in this paper. 
\end{itemize}
The adaption problem in this setting 
was studied by \cite{Ba71,GM80,TW80} 
many years ago, 
see \cite{CW04,No96,NW1} for 
more results. 
The recent paper \cite{KNU24} 
contains old and new results on the 
power of adaption and randomization.
If both features are combined 
(adaption \emph{and} randomization) 
then they are very useful for the recovery 
of $x \in \R^m$, see~\cite{KNW24,KW24},
though the approximation results are still far away from what is possible with arbitrary continuous information.
We refer to~\cite[Section~10]{KU-acta} for a detailed comparison of the various concepts. 

\medskip

In addition to (or instead of) linear information 
one might
have reasons to allow
\emph{some} other continuous information 
functionals. 
In some applications
one can measure a norm $\Vert x \Vert$
of $x$ directly, since it is a kind of 
``energy'' or ``intensity'' 
that one can measure. 
In the context of phase retrieval, 
it is assumed that 
only absolute values of certain linear functionals 
are available, 
see \cite{PSW20} for  
references and related problems.
As an open problem we can ask 
for complexity results if, 
instead of all linear or continuous functionals, 
we allow all
homogeneous or
convex or analytic  functionals 
as information mappings.

\medskip

Another interesting open problem is to prove lower bounds on the minimal number $n$ of continuous measurements needed to recover vectors with dimension $m\ge 3$ with high precision. 
From the Borsuk-Ulam Theorem, we only know that $n\ge 2$ is required, while Theorem~\ref{thm:main} states that $n\le \lceil \log_2(m)\rceil +1$. 
We mention in passing that with arbitrary discontinuous information, already a single measurement $\lambda : \R^m \to \R$ would be enough since the cardinalities of $\R^m$ and $\R$ are the same.

\medskip 

\begin{remark} 
Our main result is trivial if 
$$
n (m) = \lceil \log_2(m)\rceil +1  \ge  m,  
$$
i.e., if $m \le 3$. 
When we started to work on this problem 
during the summer 2024,
first in Passau and then in Cambridge,
we discussed the ``simplest case'' 
$m=3$ and $n=2$.
If the input is from the Euclidean ball 
in $\R^3$ then one can use adaption (and $n=2$)
to improve over all nonadaptive methods 
with worst case error 1; 
one can start with the functional 
$\ell_1(x) = x_1 - |x_3|$
and, depending on the outcome, continue with 
either $\ell_2(x)=x_3$ or $\ell_2(x)=x_2$. 
We leave it to the reader to think 
about the geometrical details. 
\end{remark} 

\section{Power of adaption for continuous information}
\label{sec:infinite}

We now show that
algorithms using adaptive 
continuous measurements
provide an exponential speed-up
compared to
algorithms using non-adaptive continuous measurements. 
This also holds for infinite-dimensional 
approximation problems, and 
implies various relations to other commonly used approximation quantities. 

\medskip

Assume that we have a mapping
$S\colon F \to Y$ 
from a topological space $F$ to a metric space $Y$.
We want to approximate $S(f)$ for all $f\in F$ using at most $n$ continuous measurements of $f$.
An interesting special case is that $F$ is a subset of $Y$ 
and $S(f)=f$, which 
corresponds
to approximating $f\in F$ in the norm of $Y$. 

We allow algorithms $A_n\colon F \to Y$ 
of the form
\begin{equation*}\label{eq:def-alg}
 A_n(f) \,=\, \Phi(\lambda_1(f),\hdots,\lambda_n(f))
\end{equation*}
with adaptively chosen continuous measurements $\lambda_i \colon F \to \R$ and an arbitrary reconstruction map $\Phi \colon \R^n \to Y$.
That is, the choice of $\lambda_i$ 
may depend
on the already chosen $\lambda_j$ and the computed values $\lambda_j(f)$ for $j<i$.
We consider the minimal worst-case error of algorithms that use at most $n$ continuous measurements, i.e., 
\[
 e_n^{\contada}(S)
 \,:=\, \inf_{A_n} \, \sup_{f\in F}\, d_Y\big(S(f),A_n(f)\big),
\]
where the infimum ranges over all $A_n$ as above.

\medskip

We start by comparing the error numbers $e_n^{\contada}(S)$ with the \emph{manifold widths} of $S$, which are defined by 
    \[
    \delta_n(S) \,:=\,  
    \inf_{\substack{N\in C(F,{\R^n}) \\ \Phi\in C(\R^n,Y)}}\, 
    \sup_{f\in F}\, d_Y\big(S(f),\Phi(N(f))\big),
    \]
    where $C(M,N)$ denotes the class of continuous mappings between topological spaces $M$ and $N$. 
These may be viewed as the minimal worst-case errors of \emph{continuous} algorithms using \emph{non-adaptive} continuous measurements. 

\begin{remark}
The quantities $\delta_n(S)$ and slight variations thereof are also known under the names of \emph{Aleksandrov widths} 
and \emph{continuous non-linear widths}.
We choose to call them manifold widths since this 
seems to be the most common name nowadays. The 
manifold widths are most often considered in the 
special case $F\subset Y$ and $S(f)=f$ as 
mentioned above, in which case they are denoted by 
$\delta_n(F)_Y$ instead of $\delta_n(S)$.
We refer to \cite{Al56,DHM,Ste74,Ste75} for 
early work as well as \cite{CDPW} for 
a more recent contribution.
Also note that for linear and bounded operators $S\colon X \to Y$ between Banach spaces (in which case $F$ is assumed to be the unit ball of $X$), the manifold widths form a scale of s-numbers, see \cite[Thm.~1]{Mathe90}.
\end{remark}

As a consequence of our 
finite-dimensional result,
we obtain the 
following upper bound.

\begin{theorem}\label{thm:mani}
For any mapping $S\colon F \to Y$ from a topological space $F$ to a metric space $Y$ and any 
$n\ge2$, it holds that
\[
 e_n^{\contada}(S)
 \;\le\; \delta_{2^{n-2}}(S).
\]
\end{theorem}

\medskip

\begin{proof}
 Let $d_Y$ be the metric on $Y$, $\varepsilon >0$ and let $m=2^{n-2}$.
 Consider continuous mappings $N\colon F \to \R^m$ and $\Phi\colon \R^m \to Y$
 such that 
 \[
  \forall f\in F \colon \quad 
  d_Y\bigl(S(f),\, \Phi(N(f))\bigr) 
  \,\le\, \delta_{m}(S)+\eps.
\]
 We construct an algorithm that uses $n$ adaptive continuous measurements and computes $S(f)$ for $f\in F$ with precision $ \delta_m(S)+2\eps$.

 As a first step, we take the continuous measurement $R=\Vert N(f) \Vert_\infty$.
 The function $\Phi$ is uniformly continuous on the cube $Q \subset \R^m$ with center zero and radius $R+1$ and hence there is some $\delta \in (0,1)$ with
 \[
 \forall y,\tilde y\in Q \colon \quad \Vert y - \tilde y \Vert_\infty \le \delta 
 \quad\Rightarrow\quad
 d_Y\bigl(\Phi(y),\, \Phi(\tilde y)\bigr) 
 \le \eps.
 \]

 As a second step, we apply Theorem~\ref{thm:main}
 to learn $N(f)$ with precision $\delta$ using at most $n(m+1)$ continuous measurements of $N(f)$, 
 i.e., with the functionals 
 $\lambda_J(N(f))$
 and $\lambda^*(N(f))$ described in 
 Section~\ref{sec:intro}.
 Since $N$ is continuous, those are continuous measurements of $f$.
 Since both $y:=N(f)$ and our approximation $\tilde y:=R_m^\delta(N(f))$ of $N(f)$ are contained in $Q$,
we approximate $\Phi(N(f))$ with precision $\eps$ by $\Phi(R_m^\delta(N(f)))$.
 This means that we have learned $S(f)$
 with an error of at most $\delta_m(S)+2\eps$
 with only $n(m+1)+1 = \lceil \log_2(m+1)\rceil +2 $ continuous measurements of $f$.

 Above, we refrained from the use of $\lambda^{**}$ for reducing from $n(m+1)$ to $n(m)$ measurements 
 in order not to increase the Lipschitz constant of $N$; see Remark~\ref{rem:Lip}.
 However, we can use another trick to reduce from $n(m+1)$ to $n(m)$.
 Indeed,
 a closer look at the coloring used in the algorithm shows that only $m$ instead of $m+1$ colors occur on the boundary of the unit cube.
 Thus, if we choose the scaling parameter for the coloring such that $R$ is a multiple, the first measurement $R=\Vert N(f) \Vert_\infty$ already excludes one of the colors and $n(m)$ instead of $n(m+1)$ measurements suffice in the second step.
 This results in a total of 
$n=\lceil \log_2(m)\rceil +2$ 
measurements.
\end{proof}

\medskip

The main purpose of this section is to compare $e_n^{\contada}(S)$ with the error of (not necessarily continuous) algorithms using only \emph{non-adaptive} continuous measurements, i.e., with
\[
 e_n^{\contnon}(S)
 \,:=\, \inf_{\substack{N\in C(F,\R^n) \\ \Phi
 \colon \R^n \to Y}}\, 
     \sup_{f\in F}\, d_Y\big(S(f),\Phi(N(f))\big).
\]
For this, we need the following lemma,
which is similar to \cite[Lemma~2.1]{DKLT93}.

\medskip

\begin{lemma}\label{lem:mani2}
For any continuous mapping $S\colon F \to Y$ from a compact metric space $F$ to a normed space $Y$ and any 
$n\in\N$, it holds that
    \[
     \delta_n(S) \,\le\, 2\cdot e_n^{\contnon}(S).  
    \]
\end{lemma}

\begin{proof}
Let $\varepsilon>0$. 
Choose a continuous mapping $N\co F\to\R^n$
and an arbitrary mapping $\Phi_0 \co \R^n \to Y$ such that
\[
 \sup_{f\in F}\,
 \Vert S(f) - \Phi_0(N(f))\Vert_Y
 \,\le\, e_n^{\contnon}(S) + \varepsilon.
\]
We have to find a continuous mapping $\Phi \co \R^n \to Y$ that satisfies a similar inequality.
By the triangle inequality, the above inequality implies 
\[ 
 \sup_{\substack{f,g\in F: \\ N(f)=N(g)}}
 \Vert S(f) - S(g)\Vert_Y
 \,\le\, 2e_n^{\contnon}(S) + 2\varepsilon
 \,=:\, \Delta. 
\]
We consider $N$ as a mapping between the compact spaces $F$ and $N(F)$
and $S$ as a mapping between the compact spaces $F$ and $S(F)$.
We first show that there is 
a covering $\{C_i\}_{i=1}^M$ of $N(F)$ with non-empty open sets such that 
$\diam(S(N^{-1}(C_i)))\le \Delta + 2\varepsilon$ 
for $i=1,\dots,M$. 

Since $S$ is uniformly continuous, we can find $\delta>0$ such that we have $S(B_\delta(A)) \subseteq B_\varepsilon(S(A))$ for every set $A\subseteq F$. Here, $B_\delta(A)$ denotes the open $\delta$-neighborhood of a set $A$.
For any $y\in N(F)$, we can write $N^{-1}(y)$ as the 
intersection of all sets $N^{-1}
(\overline{B_\eta(y)})$ for $\eta>0$. 

In particular, we have the packing
\[
 B_\delta(N^{-1}(y)) \,\supseteq\, \bigcap_{\eta>0} N^{-1}(\overline{B_\eta(y)}).
\]
By compactness of $F \setminus B_\delta(N^{-1}(y))$ and openness of $F \setminus N^{-1}(\overline{B_\eta(y)})$,
there is a finite subpacking, such that we find some $\eta(y)>0$ with 
\[
 B_\delta(N^{-1}(y)) \,\supseteq\, N^{-1}(\overline{B_{\eta(y)}(y)})
 \,\supseteq\, N^{-1}(B_{\eta(y)}(y)).
\]

On the other hand, since $N(F)$ is compact and covered by the open sets $B_{\eta(y)}(y)$ for $y\in N(F)$, there are $y_1,\hdots,y_M \in N(F)$ such that
\[
 N(F) \,\subseteq\, \bigcup_{i=1}^M B_{\eta(y_i)}(y_i).
\]
We put $C_i = B_{\eta(y_i)}(y_i)$. 
Then, indeed, the sets $\{C_i\}_{i=1}^M$ form a covering of $N(F)$ and, since 
\[
 S(N^{-1}(C_i)) \,\subseteq\, S(B_\delta(N^{-1}(y_i)))
 \,\subseteq\, B_\varepsilon(S(N^{-1}(y_i))),
\]
where the sets $S(N^{-1}(y_i))$ have a diameter at most $\Delta$, it holds that $\diam(S(N^{-1}(C_i)))\le \Delta + 2\varepsilon$, as claimed.

Now we define a continuous mapping $\Phi \colon \R^n \to Y$ such that
\begin{equation}\label{eq:cont-phi}
\sup_{f\in F}\|S(f) - \Phi(N(f))\|_Y \le  \Delta + 2\eps.
\end{equation}
For each $i\le M$, we choose some $g_i\in S(N^{-1}(C_i))$.
For $y\in\R^n$, we put 
\[
\Phi(y) \,:=\, \sum_{i=1}^M \frac{\dist(y,N(F)\setminus C_i)}{\sum_{j=1}^M \dist(y,N(F)\setminus C_j)}\cdot g_i,
\]
where the distance in $\R^n$ is measured in an arbitrary norm.
Since the denominator is always positive, the mapping $\Phi$ is continuous.
Let now $f \in F$.
For all $i\le M$ with $\dist(N(f),N(F)\setminus C_i)\neq0$ we have $N(f)\in C_i$
and hence $S(f) \in S(N^{-1}(C_i))$.  
This implies $\Vert S(f) - g_i \Vert_Y \le \Delta + 2\varepsilon$. 
Since $S(f) - \Phi(N(f))$  is a convex combination of such differences $S(f) - g_i$, we obtain~\eqref{eq:cont-phi}.
This implies \[
\delta_n(S)
\,\le\, \Delta + 2\varepsilon \,=\, 2e_n^{\contnon}(S) + 4\varepsilon,
\]
and since $\eps>0$ was arbitrary, the proof is complete. \\
\end{proof}

As an immediate consequence of Theorem~\ref{thm:mani} and Lemma~\ref{lem:mani2}, 
we get the following general statement on the advantage of continuous adaptive over continuous non-adaptive information. 

\begin{theorem}\label{thm:ada-main}
For any continuous mapping $S\colon F \to Y$ from a compact metric space $F$ to a normed space $Y$ and any $n\ge2$, it holds that
\[
 e_n^{\contada}(S)
 \;\le\; 2 \cdot e_{2^{n-2}}^{\contnon}(S).
\]
\end{theorem}

\medskip

We also want to compare $e_n^{\rm cont}(S)$ and $e_n^{\contnon}(S)$ for bounded linear operators $S\colon F \to Y$, where $F$ is assumed to be the unit ball of a Banach space $X$ and $Y$ is a Banach space.
This setting is not covered by the previous theorem since $F$ is not compact. More precisely, it is not compact, if $X$ is infinite-dimensional
and $F$ is equipped with the metric induced by 
$X$; if we consider continuous measurements with 
respect to a different metric, the setting might 
be covered by Theorem~\ref{thm:ada-main}.

\medskip

A helpful tool in the analysis of the aforementioned setting are the Bernstein numbers
\[
 b_n(S) 
 \,:=\, 
 \sup_{\substack{V_n\subset X\\\dim(V_n)=n+1}}\, \inf_{\substack{f\in V_n\\ \Vert f \Vert_X = 1}} \, \|S(f)\|_Y
\]
and the Kolmogorov numbers
\[
 d_n(S)
 \,:=\, \inf_{\substack{V_n\subset Y\\\dim(V_n)\le n}}\,
 \sup_{\substack{f\in X\\ \Vert f \Vert_X \le 1}}\,\inf_{g\in V_n} \Vert S(f) - g \Vert_Y.
\]
From the Borsuk-Ulam theorem, it follows that
\begin{equation}\label{eq:bn}
 e_n^{\contnon}(S) \,\ge\, b_n(S).
\end{equation}
Indeed, for a continuous measurement map $N\colon F \to \R^n$ and any $V_n \subset X$ of dimension $n+1$, the Borsuk-Ulam theorem gives some $f \in V_n$ with $\Vert f \Vert_X =1$ and $N(f)=N(-f)$,
such that any algorithm based on $N$ makes an error of at least $\Vert S(f) \Vert_Y$.
On the other hand, at least for strictly convex $Y$, it is not hard to show that
\begin{equation}\label{eq:dn}
 \delta_n(S) \,\le\, d_n(S).
\end{equation}
The proof uses that, for strictly convex $Y$ and any finite dimensional subspace $V_n \subset Y$, 
there is a unique continuous operator $P\colon Y\to V_n$ that maps every $y\in Y$ to its best approximation from $V_n$.
Hence,
\begin{equation}\label{eq:bnendn}
 b_n(S) \,\le\, e_n^{\contnon}(S) \,\le\, \delta_n(S) \,\le\, d_n(S).
\end{equation}
We add that, up to a constant that only depends on $Y$, \eqref{eq:dn} remains true for any space $Y$ that admits an equivalent strictly convex norm. 
According to~\cite{Day55}, 
this holds for every separable Banach space, 
but also for any $L_\infty(\mu)$-space with a $\sigma$-finite measure.

If $X$ and $Y$ are Hilbert spaces (in which case $Y$ is strictly convex), the Bernstein numbers and Kolmogorov numbers coincide, see \cite{Pietsch-s}, so that we have
\[
 b_n(S) \,=\, e_n^{\contnon}(S) \,=\, \delta_n(S) \,=\, d_n(S).
\]
With Theorem~\ref{thm:mani}, this implies the following.

\medskip

\begin{coro}\label{cor:power-of-adaption}
    Let $F$ be the unit ball of a Hilbert space, $Y$ be a Hilbert space, and $S\colon F \to Y$ be a bounded linear operator. Then, 
    \[
     e_n^{\contada}(S) \,\le\, e_{2^{n-2}}^{\contnon}(S).
    \]
\end{coro}

Hence, adaption provides an exponential speed-up for the approximation of bounded linear operators between Hilbert spaces if continuous information is allowed.

\medskip

We do not know whether the inequality in Corollary~\ref{cor:power-of-adaption} remains true in the more general case that $S\colon F \to Y$ is a linear and bounded operator from a unit ball of a Banach space to a Banach space.
However, if we assume that $e_n^{\contnon}(S) \in \mathcal{O}(n^{-\alpha})$ for some $\alpha>1$, 
then by \eqref{eq:bn} we have the same decay for $b_n(S)$. 
A theorem by Pietsch, see \cite{Pie07,U24}, implies that $d_n(S)\in \mathcal{O}(n^{-\alpha+1})$, and by \eqref{eq:dn} the same holds for $\delta_n(S)$.
Thus, by Theorem~\ref{thm:mani}, the numbers $e_n^{\contada}(S)$ decay exponentially.
Hence, at least under the assumption that $e_n^{\contnon}(S) \in \mathcal{O}(n^{-\alpha})$ for some $\alpha>1$, 
and that $Y$ has an equivalent strictly convex norm,
adaption leads to an exponential speed-up also in the general case of Banach spaces.

\medskip

\begin{remark}[Lipschitz measurements]\label{rem:Lip}
    In analogy to $e_n^{\contada}(S)$, if $F$ is a metric space,
    one can define the minimal worst-case error $e_n^{\rm Lip}(S)$
    that can be achieved with $n$
    adaptive Lipschitz-continuous measurements,
    say, with a Lip\-schitz constant at most one.
    Revisiting the proofs shows that 
    the numbers $e_n^{\rm Lip}(S)$ 
    can be bounded above by the error of any algorithm $\Phi\circ N$ with Lipschitz-continuous information $N\colon F \to \R^{2^{n-2}}$ and continuous $\Phi$.
    In particular, Theorem~\ref{thm:mani} remains valid if we replace $e_n^{\contada}(S)$ with $e_n^{\rm Lip}(S)$ and $\delta_{2^{n-2}}(S)$ 
    with the \emph{stable manifold widths} as defined in \cite{CDPW}, i.e., 
    \[
    \delta^*_{n,\gamma}(S) \,:=\,  
    \inf_{\substack{N\in \Lip_\gamma(F,{\R^n}) \\ \Phi\in \Lip_\gamma(\R^n,Y)}}\, 
    \sup_{f\in F}\, d_Y\big(S(f),\Phi(N(f))\big),
    \]
    where $\Lip_\gamma(M,N)$ denotes the class of Lipschitz-continuous mappings with Lipschitz constant $\gamma>0$ between metric spaces $M$ and $N$.
    Namely, it holds that
\begin{equation}
 e_n^{\rm Lip}(S)
\;\le\; \inf_{\gamma>0}\, \delta^*_{2^{n-2},\gamma}(S).  
\end{equation}
\end{remark}

\medskip

\medskip

\noindent
{\bf Acknowledgement.} \
We thank 
Matthieu Dolbeault 
for sharing with us the idea to reduce from $n(m+1)$ to $n(m)$ measurements.
The authors would also like to thank 
the Isaac Newton Institute for Mathematical Sciences, Cambridge, for support and hospitality during the programme 
\textit{Discretization and recovery in high-dimensional spaces}, where work on this paper was undertaken. This work was supported by EPSRC grant EP/Z000580/1.

\medskip


\linespread{0.9}



\bigskip

\noindent
\address{D.K., 
University of Passau; 
\texttt{david.krieg@uni-passau.de};  \\
E.N., 
Friedrich Schiller University Jena; 
\texttt{erich.novak@uni-jena.de}; \\
M.U., 
Johannes Kepler University Linz;
\texttt{mario.ullrich@jku.at}
}


\begin{thebibliography}{77}

\bibitem{Al56}
P. Aleksandrov, {\it Combinatorial Topology}, Vol. 1, Graylock Press, Rochester, NY, 1956.


\bibitem{Ba71}
N. S. Bakhvalov,
{\it On the optimality of linear methods for operator 
approximation in convex classes of functions}, 
USSR Comput. Maths. Math. Phys. 11, 244--249, 1971. 




















\bibitem{CDPW}
A. Cohen, R. DeVore, G. Petrova, P. Wojtaszczyk,
{\it Optimal stable nonlinear approximation},
Found. Comput. Math., 22:607--648, 2022.



\bibitem{CW04} 
J. Creutzig, P. Wojtaszczyk,
{\it Linear vs. nonlinear algorithms for linear problems},
J. Complexity, 20, 807--820, 2004. 


\bibitem{Day55}
M.M. Day, Strict convexity and smoothness of normed spaces. Transactions of the American Mathematical Society 78 (1955), 516-528.

\bibitem{DHM}
R. A. DeVore, R. Howard, and C. Micchelli, 
{\it Optimal nonlinear approximation}, 
Manuscripta Mathematica, 63:469--478, 1989.

\bibitem{DKLT93}
R.A. DeVore, G. Kyriazis, D. Leviatan, V. M. Tikhomirov, 
{\it Wavelet compression and nonlinear $n$-widths}. 
Adv. Comput. Math. 1, 197--214, 1993. 








\bibitem{GM80}
S. Gal and C. A. Micchelli, 
{\it Optimal sequential and non-sequential 
procedures for evaluating a functional},
Appl. Anal. 10, 105--120, 1980. 
















	










\bibitem{KNU24} 
D. Krieg, E. Novak, and M. Ullrich,
{\it On the power of adaption and randomization},
Forum of Mathematics, Sigma, 13:e152, 2025.

\bibitem{KU-acta}
D. Krieg and M. Ullrich, 
{\it Approximation of Functions: Optimal Sampling and Complexity}, 
Acta Numerica, 35:273--457, 2026.

\bibitem{KNW24}
R. J. Kunsch, E. Novak, and M. Wnuk,
\textit{Randomized approximation of summable sequences -- adaptive and non-adaptive}, 
J. Approx. Theory 304, 106056, 2024.

\bibitem{KW24}
R. J. Kunsch, M. Wnuk,
\textit{Uniform approximation of vectors using 
adaptive randomized information}, 
J. Approx. Theory 313, 106216, 2026.

\bibitem{Mathe90}
P. Math\'e, \textit{s-numbers in information-based
complexity}, J. Complexity 6, 41--66, 1990.

\bibitem{Novak92} E. Novak. {\it Optimal linear randomized methods for linear operators in Hilbert
spaces}, J. Complexity 8, 22--36, 1992.

\bibitem{No96}
E. Novak, 
\textit{On the power of adaption},  
J. Complexity 12, 199--237, 1996. 

\bibitem{NW1}
E.~Novak and H.~Wo\'zniakowski,
\newblock {\em Tractability of multivariate problems. {V}olume~{I}: {L}inear
information}, volume~6 of {\em EMS Tracts in Mathematics}.
\newblock European Mathematical Society (EMS), Z\"urich, 2008.

\bibitem{Pears} 
A.\,R.\,Pears,
{\em Dimension theory of general spaces.}
Cambridge University Press, 
1975.

\bibitem{Pietsch-s}
A. Pietsch, {\it s-numbers of operators in Banach spaces}, 
Studia Math. 51, 201--223, 1974.




\bibitem{Pie07} 
A. Pietsch, {\it History of Banach spaces and linear operators}, Birkhäuser Boston, MA, 2007.




\bibitem{PSW20}
L. Plaskota, P. Siedlecki and 
H. Wo\'zniakowski,
{\it Absolute value information for IBC problems}.
J. Complexity 56 (2020), 101427. 


\bibitem{Ste74}
M. I. Stesin, 
{\it On Aleksandrov diameters of balls}, 
Dokl. Akad. Nauk SSSR 217, 1, 31--33, 1974. 

\bibitem{Ste75}
M. I. Stesin, 
{\it Aleksandrov diameters of finite-dimensional 
sets and classes of 
smooth functions}, 
Dokl. Akad. Nauk SSSR 220, 6, 1278--1281, 1975. 

\bibitem{TW80}
J.~Traub and H.~Wo\'zniakowski, 
\emph{A General Theory of Optimal Algorithms},
Academic Press, 1980.

\bibitem{U24} M. Ullrich, {\it Inequalities between s-numbers}, Adv. Oper. Theory 9, 82 (2024). 

\end{thebibliography}
\end{document}